\documentclass[a4paper,12p]{article}
\usepackage{latexsym}
\usepackage[pdftex]{graphicx}
\usepackage{amsmath, amsthm, amsfonts, amssymb}
\usepackage[bookmarksnumbered, plainpages]{hyperref}

\title{A construction of fractal surfaces with function scaling factors on a rectangular grid}

\author{Chol-Hui Yun,$~~$\textrm{Hui-Chol Choi}, $~~$\textrm{Hyong-Chol O}\\ \\
                    {\small\textit{Faculty of Mathematics }}\\
                        {\small\textit{\textbf{Kim Il Sung} University, Pyongyang, D.P.R.Korea}}}

\date{}  
\begin{document}
\maketitle      
\begin{abstract}   
A fractal surface is a set which is a graph of a bivariate continuous function. 
In the construction of fractal surfaces using IFS, vertical scaling factors in IFS are important one which characterizes a fractal 
feature of surfaces constructed. 
We construct IFS with function vertical scaling factors which are 0 on the boundaries of a rectangular grid using arbitrary data 
set on a rectangular grid and give a condition for an attractor of the IFS constructed being a surface. 
Finally, lower and upper bounds of Box-counting dimension of the constructed surface are estimated.
\end{abstract}
\section{Introduction}
A fractal surface is a fractal set which is a graph of a bivariate continuous function and has been widely applied in the approximation theory, 
computer graphics, image compression, metallurgy, physics, and geography and so on. 
Barnsley \cite{ban3} defined a fractal interpolation function (FIF) in 1986, which is one whose graph is an attractor of  IFS, 
and presented a construction of fractal curves by fractal interpolation, which has been widely used in many papers \cite{bou}-\cite{mal}, \cite{mas}-\cite{zha}.

By Massopust \cite{mas} was introduced the constructions of fractal surfaces on triangular data set at which the interpolation points 
on the boundary data are coplanar. 
In \cite{ger}, this construction was generalized to allow more general data set.

In many papers, constructions of fractal interpolation surfaces with a data set on a rectangular grid were studied in \cite{bou, dal, fen, mal, met}. 
Dalla \cite{dal} presented the construction of  the fractal surfaces on rectangular data set in which interpolation points on the 
boundary data were collinear. 
By Malysz \cite{mal}, this construction was generalized, where the IFS was constructed with constant vertical scaling factors, 
linear domain contraction mapping and quadratic polynomials. 
And all of the contraction transformations of the IFS constructed had the same vertical scaling factor.

In \cite{met}, this construction was generalized to allow the arbitrary data set, function vertical scaling factors, Lipschitz domain 
contraction mapping and Lipschitz mappings, and lower and upper bounds for the Box-counting dimension of the fractal 
surfaces constructed were estimated. 
This estimation was improved in \cite{yun1}. 
The fractal surfaces were constructed using a recurrent iterated functions system (RIFS) with function vertical scaling factors and 
the Box-counting dimension of the fractal surfaces constructed were estimated in \cite{yun2}. 
Vertical scaling factors in IFS are important one that characterizes the fractal set constructed. 
It is more general that fractal sets in nature have different scaling factors at each point. 
Therefore, we can model better fractal sets in nature using IFS in which vertical scaling factors are functions.

In the construction of the fractal surfaces, interpolation functions of the data points on the end points of  the rectangular and of  the 
given data set  were used in \cite{met}, while interpolation functions of the data points on the end points of  the domains and of  
the given data set were used in \cite{yun2}. 
Feng et al. \cite{fen} presented a construction of fractal surfaces on a new condition that the function vertical scaling factors are 0 
on the grid of rectangular domain, but did not estimate a fractal dimension of surface constructed.

In this paper, we generalize the construction of fractal surfaces in \cite{fen}. 
First, we construct IFS with function vertical scaling factors being 0 on the grid of rectangular domain, 
any Lipschitz functions on the rectangular and interpolation functions of  the given data set. 
Finally, lower and upper bounds for the Box-counting dimension of the constructed surfaces are estimated.    
  
                                            
\section{Construction of fractal surfaces}
In this section, we construct IFS that function vertical scaling factors on the basis of a data set on a rectangular grid and present a condition for an attractor of the IFS constructed being surface.
   
         
\subsection{IFS with function vertical scaling factors}

Let a data set on a rectangular grid be 
\begin{displaymath}
 P=\left\{(x_i,y_j, z_{ij})\in \mathbf{R}^3;i=0, 1,\ldots, n, j=0, 1, \ldots, m \right\} 
\end{displaymath}
\begin{displaymath}
(x_0<x_1<\ldots <x_n, ~ y_0<y_1<\ldots <y_m),
\end{displaymath}
and denote  
\begin{eqnarray}
&&I_{x_{i}}=[x_{i-1},x_{i}],~I_{y_{j}}=[y_{j-1},y_{j}],I_{x}=[x_{0},x_{n}],~I_{y}=[y_{0},y_{m}],\nonumber \\ 
&&N_{nm}=\left\{1,\ldots,n\right\}\times \left\{1,\ldots,m\right\},
~E_{ij}=I_{x_{i}}\times I_{y_{j}},~E=I_{x}\times I_{y},\quad\hbox{ for $\left(i,j\right)\in N_{nm}$,}\nonumber \\  
&& P_{x_{\alpha}}=\left\{\left(x_{\alpha},y_{l},z_{\alpha l}\right)\in P;~l=0,1,\ldots,m~\right\},\ \ \ \alpha =0,1\ldots,n \nonumber \\
&& P_{y_{\beta}}=\left\{\left(x_{k},y_{\beta},z_{k\beta}\right)\in ~\mathrm{P};~k=0,1,\ldots,n~\right\},\ \ \ \beta =0, 1\ldots,m \nonumber 
\end{eqnarray}
We define contraction transformations $L_{ij}:E \rightarrow  E_{ij}, ~\left(i,j\right)\in {N}_{nm}$   by
\begin{displaymath}
 L_{ij}(x, y)=(L_{x_i}(x), ~L_{y_j}(y)), ~(x, y)\in E,
\end{displaymath}
where  $L_{x_i}:I_x \rightarrow I_{x_i}, ~L_{y_j}:I_y \rightarrow I_{y_j}$ are contraction homeomorphisms satisfying the following conditions;
\begin{displaymath}
L_{x_i}(x_\alpha )=x_a, ~L_{y_j}(y_\beta )=y_b, ~(\alpha, ~\beta )\in \left\{0, n\right\}\times \left\{0, m\right\}, (a,~b)\in \left\{i-1,i\right\}\times \left\{j-1, j\right\}
\end{displaymath}
And a function  $F_{ij}:E\times \mathbf{R}\rightarrow \mathbf{R}$  is defined by
\begin{eqnarray}
F_{ij}(x,y,z)=s_{ij}(L_{ij}(x,y))z+Q_{ij}(x,y),   \label{eq1}                  
\end{eqnarray}
where $s_{ij}:E_{ij}\rightarrow \mathbf{R}$  is  a Lipschitz function such that $|s_{ij}(x,y)|<1,(x,y)\in E_{ij}$  , which is called a function vertical scaling factor,  $Q_{ij}:E \rightarrow \mathbf{R}$ is a Lipschitz function.  A Lipschitz (contraction) constant of a Lipschitz (contraction) function is denoted by $L_f(c_f)$ . 

Transformations $W_{ij}:E\times \mathbf{R} \rightarrow E_{ij}\times \mathbf{R},~i=1,\ldots,n,~j=1,\ldots,m$   are defined by
\begin{displaymath}
W_{ij}(x,y,z)=(L_{ij}(x,y),F_{ij}(x,y,z)).
\end{displaymath}

For $\theta (0<\theta<(1-\bar c_L)/\bar L_Q)$ , let define a metric $ \rho _\theta$  on  $\mathbf{R}^3$ as following
 
\begin{displaymath}
\rho _\theta ( (x,y,z), (x^\prime, y^\prime, z^\prime ))=|x-x^\prime | +|y-y^\prime |+\theta |z-z^\prime |, ~ (x,y,z), (x^\prime, y^\prime, z^\prime )\in \mathbf{R}^3,
\end{displaymath}
where $ \bar c_L = \max \left\{c_{L_{ij}}; i=1,\ldots, n, ~j=1,\ldots,m\right\}$ , $ \bar L_Q =\max \left\{L_{Q_{ij}}; i=1,\ldots, n, ~j=1,\ldots,m\right\}$ . We can easily prove that the metric $\rho _\theta$  is equivalent to the Euclidean metric on  $\mathbf{R}^3$ and $W_{ij}:E\times \mathbf{R} \rightarrow E_{ij}\times \mathbf{R},~i=1,\ldots,n,~j=1,\ldots,m$  are contraction transformations with respect to the metric $\rho _\theta$ .

Now, we have an IFS $\left\{\mathbf{R}^3; ~W_{ij},~i=1,\ldots,n,~j=1,\ldots,m \right\}$  corresponding the data set $P$.

 
\subsection{Fractal surfaces}

We present a condition for an attractor of the IFS constructed above being surface, which is a graph of a bivariate continuous function, and construct an operator generating the bivariate continuous function whose graph is the attractor.

Let the function vertical scaling factors $s_{ij}: E_{ij} \rightarrow  \mathbf{R}, ~(i,j)\in N_{nm}$ satisfy the following conditions:
\begin{eqnarray}
s_{ij}(x_{i-1}, y)=0, ~s_{ij}(x_i, y)=0, ~s_{ij}(x, y_{j-1})=0, ~s_{ij}(x, y_j)=0, \label{eq2}
\end{eqnarray}
and let define the Lipschitz functions $Q_{ij}:E \rightarrow  \mathbf{R}, ~(i,j)\in N_{nm}$ as follow;
\begin{eqnarray}
Q_{ij}(x,y)=-s_{ij}(L_{ij}(x,y))g_{ij}(x,y)+h_{ij}(L_{ij}(x,y)), \label{eq3}
\end{eqnarray}
where $g_{ij}:E\rightarrow \mathbf{R}$ are any Lipschitz functions and  $h_{ij}:E_{ij}\rightarrow \mathbf{R}$ are Lipschitz functions whose boundary curves are $q_{i-1}|_{\left\{x_{i-1}\right\}\times I_{y_j}}$, $q_{i}|_{\left\{x_{i}\right\}\times I_{y_j}}$, $r_{j-1}|_{I_{x_i}\times\left\{y_{j-1}\right\}}$, $r_{j}|_{I_{x_i}\times\left\{y_{j}\right\}}$ i.e. 
\begin{eqnarray}
h_{ij}|_{\left\{x_{i-1}\right\}\times I_{y_j}}=q_{i-1}|_{\left\{x_{i-1}\right\}\times I_{y_j}}, h_{ij}|_{\left\{x_{i}\right\}\times I_{y_j}}=q_{i}|_{\left\{x_{i}\right\}\times I_{y_j}},  \label{eq4} \\
h_{ij}|_{I_{x_i}\times\left\{y_{j-1}\right\}}=r_{j-1}|_{I_{x_i}\times \left\{y_{j-1}\right\}}, h_{ij}|_{I_{x_i}\times\left\{y_{j}\right\}}=r_{j}|_{I_{x_i}\times \left\{y_{j}\right\}},   \label{eq5}
\end{eqnarray}
where   $ q_{\alpha}:\left\{x_{\alpha}\right\}\times I_y\rightarrow \mathbf{R} $, $\alpha =0,1,\ldots,n$, $r_{\beta}:I_x\times \left\{y_{\beta}\right\}\rightarrow\mathbf{R}$, $\beta=0,1,\ldots,m$    are continuous functions interpolating data sets $P_{x_{\alpha}}$, $P_{y_{\beta}}$,  respectively.

\newtheorem{Thot}{Theorem}
\begin{Thot}\label{theo1}
 If the function vertical scaling factors $ s _{ij}(x,y):E_{ij}\rightarrow \mathbf{R}, ~(i,j)\in N_{nm}$    and the Lipschitz functions  $Q_{ij}(x,y):E\rightarrow \mathbf{R}, ~(i,j)\in N_{nm}$ in \eqref{eq1} satisfy the above conditions, then there exists a bivariate continuous function whose graph is the attractor of the IFS constructed in Sec 2.1.
\end{Thot}
\begin{proof}
Let a set  $C(E)$  be
\begin{displaymath}
C(E)=\left\{\varphi \in C^0(E): ~\varphi |_{\left\{x_{\alpha}\right\}\times I_y}=q_{\alpha}, ~\varphi |_{I_x\times \left\{y_{\beta}\right\}}=r_{\beta}, ~\alpha =0,1, \ldots , n, ~\beta =0,1,\ldots ,m\right\}.
\end{displaymath}
We can easily prove that  $(C(E), \Vert \cdot \Vert_\infty )$ is a complete metric space.

For a function  $\varphi (\in C(E)) $, a function $T \varphi:E\rightarrow \mathbf{R} $ is defined  by
\begin{eqnarray}
(T \varphi )(x,y) & = & F_{ij}(L^{-1}_{ij}(x,y),\varphi (L^{-1}_{ij}(x,y))) \nonumber \\
& = & s_{ij}(x,y)(\varphi (L^{-1}_{ij}(x,y))-g_{ij}(L^{-1}_{ij}(x,y)))+h_{ij}(x,y), ~ (x,y)\in E_{ij}.  \label{eq6}
\end{eqnarray}
Then, we have $T \varphi \in C(E)$ .

In fact, for $(x_i,y)\in \left\{x_i\right\} \times I_{y_j}$, $i\in\left\{1,\ldots , n-1\right\}$, $j\in\left\{1,\ldots,m\right\}$  we have
\begin{eqnarray}
(T\varphi )(x_i,y)=F_{ij}(L^{-1}_{ij}(x_i,y), ~ \varphi (L^{-1}_{ij}(x_i,y)))=h_{ij}(x_i,y)=q_i(x_i,y),\nonumber \\
(T\varphi )(x_0,y)=F_{1j}(L^{-1}_{1j}(x_0,y), ~ \varphi (L^{-1}_{1j}(x_0,y)))=h_{1j}(x_0,y)=q_0(x_0,y),\nonumber
\end{eqnarray}
by \eqref{eq4}. Thus,  $T \varphi |_{\left\{x_{\alpha}\right\}\times I_y}=q_{\alpha}$, $\alpha=0,1,\ldots ,n$. Similarly, we can prove that  $T \varphi |_{I_x\times \left\{y_\beta \right\}}=r_{\beta}$, $\beta=0,1,\ldots , m$  by \eqref{eq5}. This shows that  $T \varphi$ is a continuous function on $E$ and $T \varphi\in C(E)$. And, it is clear that $(T \varphi)(x_i, y_j)=z_{ij} $ , $(i,j)\in N_{nm}$.

This means that an operator $T:C(E)\rightarrow C(E)$  is defined by \eqref{eq6}. We can prove easily that the operator $T:C(E)\rightarrow C(E)$    is contraction one with respect to the norm $\Vert \cdot \Vert_\infty $ . 
Therefore,  according to the fixed point theorem in the complete metric space, the operator $T$  has a unique fixed point  $f( \in C(E))$ such that 
\begin{eqnarray*}
f(x,y) & = & F_{ij}(L^{-1}_{ij}(x,y),f(L^{-1}_{ij}(x,y)))  \\
& = & s_{ij}(x,y)(f(L^{-1}_{ij}(x,y))-g_{ij}(L^{-1}_{ij}(x,y)))+h_{ij}(x,y), ~ (x,y)\in E_{ij}, ~(i,j)\in N_{nm}.
\end{eqnarray*}
This shows that $A=Gr(f)$ , i.e.
\begin{displaymath}
Gr(f)= \bigcup_{(i,j)\in N_{nm}}W_{ij}(Gr(f)),
\end{displaymath}
where $Gr(f)$  is a graph of the function $f$.
\end{proof}

\noindent \textbf{Remark 1.} In \cite{fen}, $s_{ij}:E\rightarrow\mathbf{R}$, $s_{ij}(x,y)=\lambda_{ij}(x-x_0)(x_n-x)(y-y_0)(y_m-y) $, $g_{ij}(x,y)\equiv  0$  and  $ h_{ij}(L_{ij}(x, y))=t_{ij}xy+e_{ij}x+r_{ij}y+k_{ij}$. In \cite{met}, $g_{ij}$  are interpolation functions of the data points on the end points of the rectangular and of the data points on the end points of  the domains in \cite{yun2}.\\

\noindent \textbf{Example 1}. Let functions $d_{ij}(t):\mathbf{R}\rightarrow\mathbf{R}, ~(i,j)\in N_{nm}$  be Lipschitz functions such that $|d_{ij}(t)|<1$, $d_{ij}(0)=0$. For Lipschitz functions  $\psi _{ij}:E_{ij}\rightarrow \mathbf{R}$ and constants $\lambda_{x_i},~ \lambda_{x_{i-1}}$, $\lambda_{y_{j}}, ~ \lambda_{y_{j-1}}\in \mathbf{R},~(i,j)\in N_{nm}$, functions $t_{ij}:E_{ij}\rightarrow \mathbf{R},~(i,j)\in N_{nm}$ are defined by Lipschitz functions 
\begin{displaymath}
t_{ij}(x,y)=\psi_{ij}(x, y)(x-x_i)^{\lambda _{x_i}}(x-x_{i-1})^{\lambda _{x_{i-1}}}(y-y_j)^{\lambda_{y_j}}(y-y_{j-1})^{\lambda_{y_{j-1}}},
\end{displaymath}
and functions $s_{ij}:E_{ij}\rightarrow \mathbf{R},~(i,j)\in N_{nm}$ by $s_{ij}(x,y)=d_{ij}(t_{ij}(x,y))$. The functions $s_{ij}(x,y)$  satisfy \eqref{eq2}. \\

\noindent \textbf{Example 2}. Let a data set  $P$ be given as the Table 1 \cite{fen}:

%

\vspace{0.5cm}
\centerline{Table 1: The data set $P$}
\vspace{0.2cm}
\centerline{
\begin{tabular}{c   c   c   c   c   c}
\hline
& &  &  $x~~~$ &  & \\  \cline{2-6}
$y \quad$ & $0 \quad$ & $0.25 \quad$ & $0.5 \quad$ & $0.75 \quad$ & $1 $ \\	
\hline
$0 \quad$ & $0.3 \quad$ & $1.1 \quad$ & $0.2 \quad$ & $1.5 \quad$ & $2 $ \\
$1/3 \quad$ & $0.3 \quad$ & $2 \quad$ & $1.8 \quad$ & $1.5 \quad$ & $2 $ \\
$2/3 \quad$ & $3 \quad$ & $2 \quad$ & $3 \quad$ & $3.3 \quad$ & $3 $ \\
$1 \quad$ & $2 \quad$ & $3 \quad$ & $2.5 \quad$ & $4 \quad$ & $4.5 $ \\
\hline
\end{tabular}
}
\vspace{0.5cm}
\vspace{0.3cm}

The contraction mappings  $L_{x_i}$,  $L_{y_j}, ~i=1,2,3,4, ~j=1,2,3$ are defined by 
\begin{eqnarray}
&& L_{x_1}(x)=x/4,~L_{x_2}(x)=x/4+1/4,~L_{x_3}(x)=x/4+2/4,~L_{x_4}(x)=x/4+3/4,\nonumber\\
&& L_{y_1}(y)=y/3, ~L_{y_2}(y)=y/3+1/3, ~L_{y_3}(y)=y/3+2/3\nonumber
\end{eqnarray}
 and functions  $ q_{\alpha}:\left\{x_{\alpha}\right\}\times I_y\rightarrow \mathbf{R} $, $\alpha =0,1,2,3,4$, $r_{\beta}:I_x\times \left\{y_{\beta}\right\}\rightarrow\mathbf{R}$, $\beta=0,1,2,3$, $h_{ij}:E_{ij} \to \mathbf{R}$, $i=1,2,3,4, j=1,2,3$ are as following:
 
\begin{eqnarray*}\label{d}
q_0(0,y)=\left\{
\begin{array}{ll}
18y^2-6y+0.3, & y\in I_{y_1}\\
18y^2-9.9y+1.6, & y\in I_{y_2} \\
-18y^2+27y-7, & y\in I_{y_3} \\
\end{array}\right. 
q_1(0.25,y)=\left\{
\begin{array}{ll}
4.5y^2+1.2y+1.1, & y\in I_{y_1}\\
9y^2-9y+4, & y\in I_{y_2} \\
18y^2-27y+12, & y\in I_{y_3} \\
\end{array}\right. 
\end{eqnarray*}

\begin{eqnarray*}\label{d}
q_2(0.5,y)=\left\{
\begin{array}{ll}
8.1y^2+2.1y+0.2, & y\in I_{y_1}\\
9y^2-5.4y+2.6, & y\in I_{y_2} \\
-9y^2+13.5y-2, & y\in I_{y_3} \\
\end{array}\right. 
q_3(0.75,y)=\left\{
\begin{array}{ll}
9y^2-3y+1.5, & y\in I_{y_1}\\
9y^2-3.6y+1.7, & y\in I_{y_2} \\
-4.5y^2-5.4y+4.9, & y\in I_{y_3} \\
\end{array}\right. 
\end{eqnarray*}

\begin{eqnarray*}\label{d}
q_4(1,y)=\left\{
\begin{array}{ll}
9y^2-3y+2, & y\in I_{y_1}\\
18y^2-15y+5, & y\in I_{y_2} \\
9y^2-10.5y+6, & y\in I_{y_3} \\
\end{array}\right. 
\end{eqnarray*}
 
 \begin{eqnarray*}\label{d}
r_0(x,0)=\left\{
\begin{array}{ll}
6.4x^2+1.6x+0.3, & x\in I_{x_1}\\
-16x^2+8.4x, & x\in I_{x_2} \\
16x^2-14.8x+3.6, & x\in I_{x_3} \\
4.8x^2-6.4x+3.6, & x\in I_{x_4} \\
\end{array}\right.
r_1(x,\frac{1}{3})=\left\{
\begin{array}{ll}
8x^2+4.8x+0.3, & x\in I_{x_1}\\
-16x^2+11.2x+0.2, & x\in I_{x_2} \\
-9.6x^2+10.8x-1.2, & x\in I_{x_3} \\
4.8x^2-6.4x+3.6, & x\in I_{x_4} \\
\end{array}\right.
\end{eqnarray*}

 \begin{eqnarray*}\label{d}
r_2(x,\frac{2}{3})=\left\{
\begin{array}{ll}
-32x^2+4x+3, & x\in I_{x_1}\\
32x^2-20x+5, & x\in I_{x_2} \\
3.2x^2-2.8x+3.6, & x\in I_{x_3} \\
-16x^2+26.8x-7.8, & x\in I_{x_4} \\
\end{array}\right.
r_3(x,1)=\left\{
\begin{array}{ll}
32x^2-4x+2, & x\in I_{x_1}\\
-16x^2+10x+1.5, & x\in I_{x_2} \\
16x^2-14x+5.5, & x\in I_{x_3} \\
16x^2-26x+14.5, & x\in I_{x_4} \\
\end{array}\right.
\end{eqnarray*}

\begin{eqnarray*}
&&  h_{11}(x,y)=18y^2+4.8x^2y-54xy^2+6.4x^2+27.6xy+1.6x-6y+0.3, \\
&&  h_{12}(x,y)=18y^2-120x^2y-36xy^2+48x^2+33.6xy-2.4x-9.9y+1.6, \\
&&  h_{13}(x,y)=-18y^2+192x^2y+144xy^2-160x^2-264xy+116x+27y-7, \\
&&  h_{21}(x,y)=0.9y^2+14.4xy^2-16x^2+3.6xy+8.4x+0.3y, \\
&&  h_{22}(x,y)=9y^2+144x^2y-64x^2-93.6xy+42.4x+5.4y-2.6, \\
&&  h_{23}(x,y)=45y^2-144x^2y-108xy^2+128x^2+270xy-152x-85.5y+42,  \\
&&  h_{31}(x,y)=6.3y^2-76.8x^2y+3.6xy^2+16x^2+75.6xy-14.8x-16.5y+3.6, \\
&&  h_{32}(x,y)=9y^2+38.4x^2y-22.4x^2-40.8xy+24.4x+5.4y-4, \\
&&  h_{33}(x,y)=-36y^2+38.4x^2y+54xy^2-22.4x^2-123.6xy+55.6x+65.7y-24.2, \\
&&  h_{41}(x,y)=-18y^2+192x^2y+144xy^2-160x^2-264xy+116x+27y-7,  \\
&&  h_{42}(x,y)=-18y^2-62.4x^2y+36xy^2+25.6x^2+63.6xy-31.6x-16.2y+11,  \\
&&  h_{43}(x,y)=-9y^2+96x^2y+18xy^2-80x^2-188.4xy+144.4x+81.9y-58.4.
\end{eqnarray*} 
In a), b) of the Figure 1 and Figure 2, the function vertical scaling factors  $s_{ij}(x,y)$ , $(i,j)\in N_{nm}$ are defined as follows, respectively: 
\begin{eqnarray*}
a) && s_{11}(x,y)=2120xy(x-0.25)(y-1/3),~s_{12}(x,y)=150x(x-0.25)(y-1/3)(y-2/3), \\
&& s_{13}(x,y)=400x(x-0.25)(y-2/3)(y-1),~s_{21}(x,y)=-2111(x-0.25)(x-0.5)y(y-1/3), \\
&& s_{22}(x,y)=2300(x-0.25)(x-0.5)(y-1/3)(y-2/3),~s_{23}(x,y)=-950(x-0.25)(x-0.5)(y-2/3)(y-1), \\
&& s_{31}(x,y)=333(x-0.5)(x-0.75)y(y-1/3),~s_{32}(x,y)=-1903(x-0.5)(x-0.75)(y-1/3)(y-2/3), \\
&& s_{33}(x,y)=435(x-0.5)(x-0.75)(y-2/3)(y-1),~s_{41}(x,y)=-2123(x-0.75)(x-1)y(y-1/3), \\
&& s_{42}(x,y)=666(x-0.75)(x-1)(y-1/3)(y-2/3),~s_{43}(x,y)=-2119(x-0.75)(x-1)(y-2/3)(y-1),
\end{eqnarray*}
\begin{eqnarray*}
b) && s_{11}(x,y)=2119xy(x-0.25)(y-1/3),~s_{12}(x,y)=1580x(x-0.25)(y-1/3)(y-2/3), \\
&& s_{13}(x,y)=-2111x(x-0.25)(y-2/3)(y-1),~s_{21}(x,y)=1888(x-0.25)(x-0.5)y(y-1/3), \\
&& s_{22}(x,y)=2300(x-0.25)(x-0.5)(y-1/3)(y-2/3),~s_{23}(x,y)=-2103(x-0.25)(x-0.5)(y-2/3)(y-1), \\
&& s_{31}(x,y)=1989(x-0.5)(x-0.75)y(y-1/3),~s_{32}(x,y)=-1903(x-0.5)(x-0.75)(y-1/3)(y-2/3), \\
&& s_{33}(x,y)=2003(x-0.5)(x-0.75)(y-2/3)(y-1),~s_{41}(x,y)=-2123(x-0.75)(x-1)y(y-1/3), \\
&& s_{42}(x,y)=1673(x-0.75)(x-1)(y-1/3)(y-2/3),~s_{43}(x,y)=-2118(x-0.75)(x-1)(y-2/3)(y-1).
\end{eqnarray*}
And functions $g_{ij}:E\rightarrow\mathbf{R}$ are $g_{ij}(x,y)\equiv  0$, $g_{ij}(x,y)=\sin (\pi ^2xy)$  in the Figure 1 and the Figure 2, respectively.

The Figure 1 and the Figure 2 are fractal surfaces generated by the IFS constructed with $L_{ij}, Q_{ij}, s_{ij}$ above.

\begin{figure}
\centering
\includegraphics[width=12cm , height=20cm]{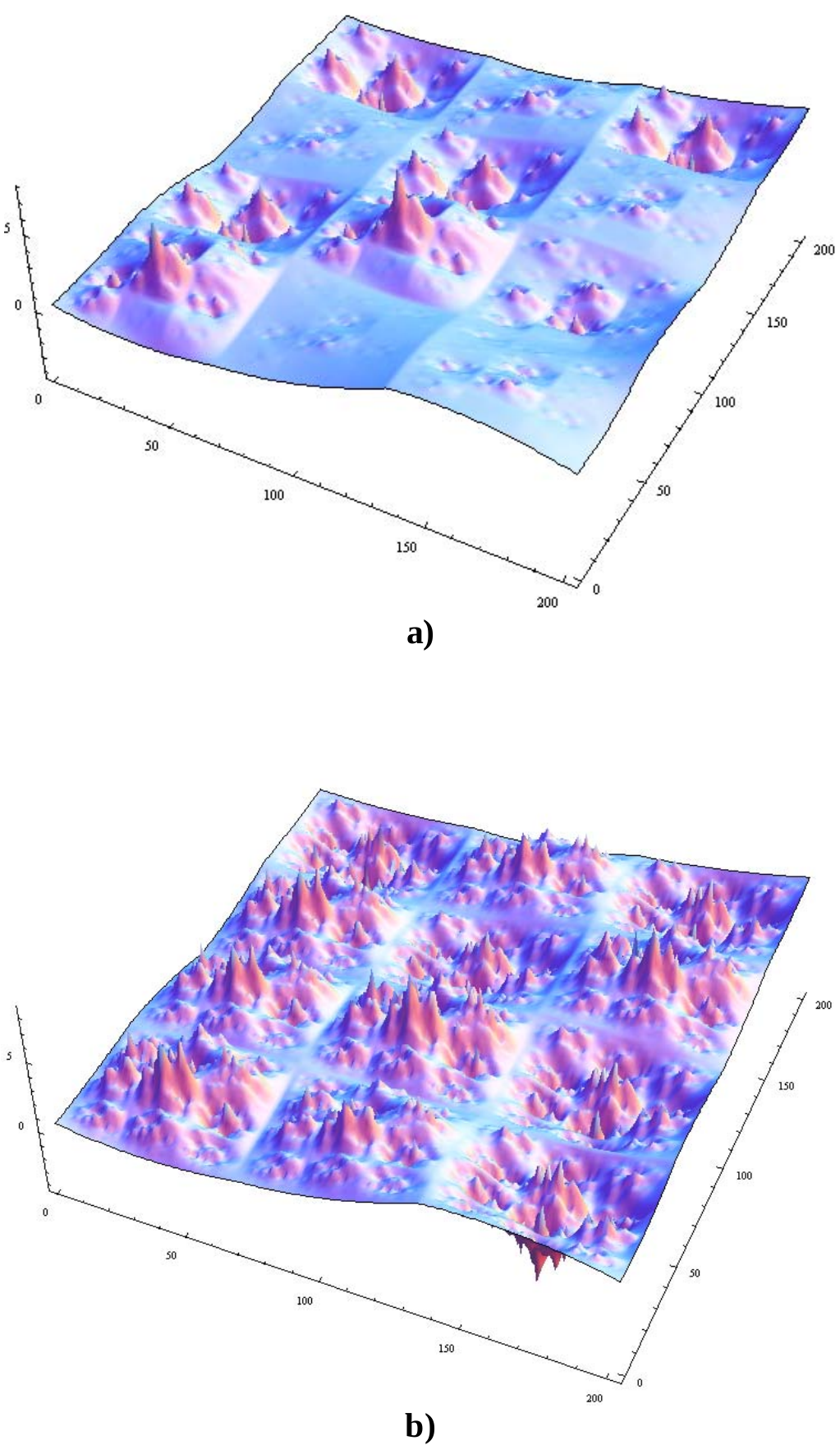}
\caption{Fractal surfaces with the function vertical scaling factors $g_{ij}(x,y)\equiv  0$}
\end{figure}

\begin{figure}
\centering
\includegraphics[width=12cm , height=20cm]{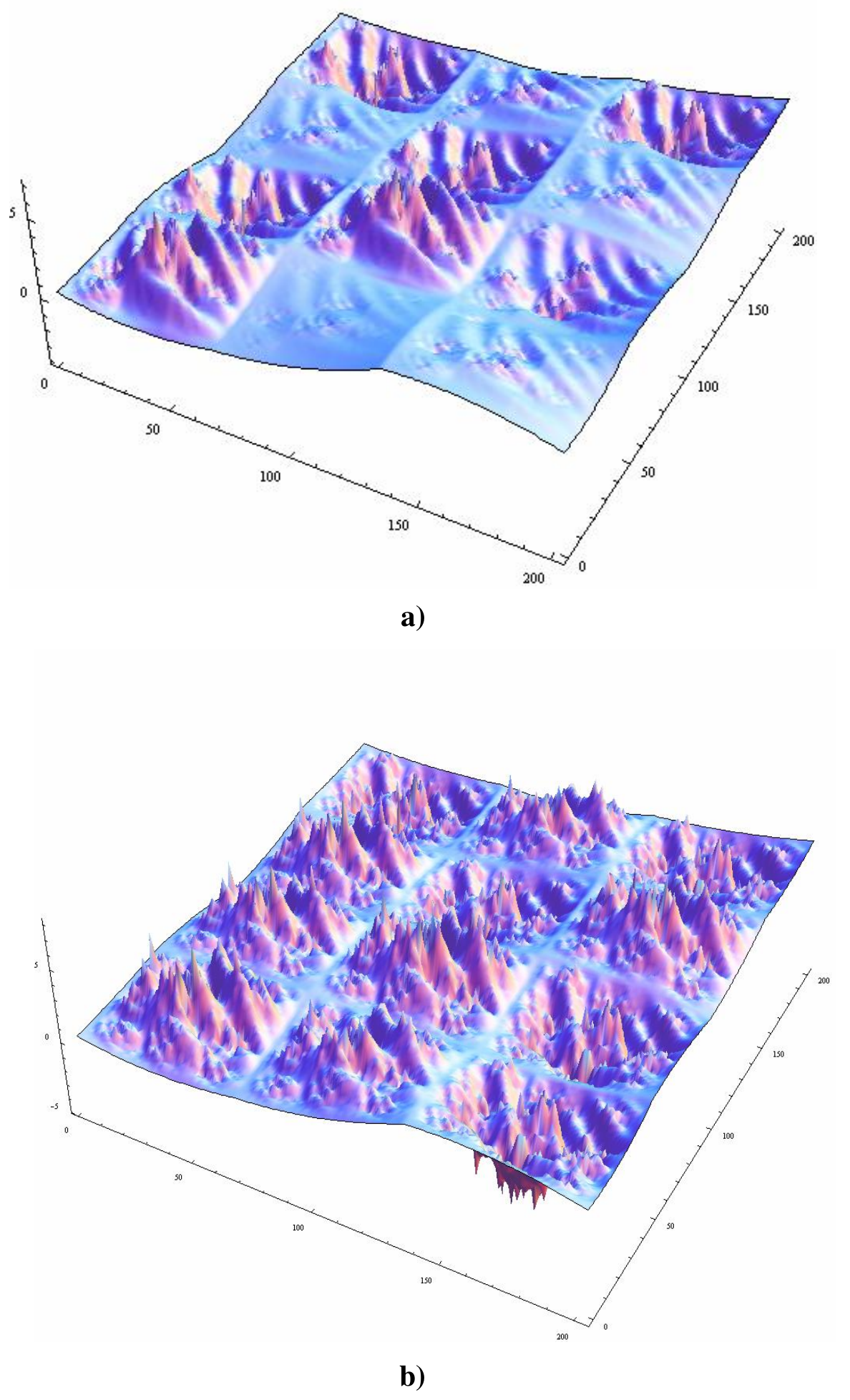}
\caption{Fractal surfaces with the function vertical scaling factors $g_{ij}(x,y)=\sin(\pi ^2xy)$}
\end{figure} 

\noindent \textbf{Remark 2} We can generate $h_{ij}(x,y)~(i,j)\in N_{nm}$  using the construction of fractal surfaces in \cite{yun3}, which gives us the fractal surfaces that are not smooth on the grids of rectangular domain.


\section{Box-counting dimension of the fractal surface.}
In this section, we estimate Box-counting dimension of the surfaces constructed in the Sec. 2 with the data set 
\begin{displaymath}
P=\left\{\left(x_0+\frac{x_n-x_0}{n}i, ~y_0+\frac{y_n-y_0}{n}j,~z_{ij}\right)\in \mathbf{R}^3;~i,j=0,\ldots,n  \right\}
\end{displaymath}

Since there is a bi-Lipschitz mapping between arbitrary rectangular $[a,b]\times[c,d]$  in  $\mathbf{R}^2$ and $[0,1]\times[0,1]$ , and Box-counting dimension is invariant under the bi-Lipschitz mapping, we can suppose that $E=[0,1]\times[0,1]$. Then the data set  $P$ is as follows;
 \begin{displaymath}
P=\left\{\left(\frac{i}{n}, ~\frac{j}{n},~z_{ij}\right)\in \mathbf{R}^3;~i,j=0,\ldots,n  \right\}.
\end{displaymath}

Let denote $\mathring{E}_{ij}:=E_{ij} \setminus \partial E_{ij}$, $\bar{s}_{ij}:=\max_{\mathring{E}_{ij}}\left\vert s_{ij}(x,y) \right\vert$,     
 $\underline{s}_{ij}:=\min_{\mathring{E}_{ij}} \left\vert s_{ij}(x,y) \right\vert (>0)$   and the surface constructed in the Theorem \ref{theo1} by $A(=Gr(f))$. \\

\begin{Thot}\label{gen}
 If there exists  $\alpha_0(\in \{1,\ldots,n-1\})$ $($or $\beta_0(\in \{1,\ldots,n-1\}))$ such that $P_{x_{\alpha_0}} ($or $P_{y_{\beta_0}})$ are not collinear, then box-counting dimension  $ \mathrm{dim}_BA$ of the surface $A$  is as follows:

$\mathrm{1)} $ If $\sum_{i,j=1}^n\underline s_{ij}>n$ , then
 \begin{displaymath}
 1+\log_n^{\sum_{i,j=1}^n\underline s_{ij}} \leq  \mathrm{dim}_BA \leq 1+\log_n^{\sum_{i,j=1}^n\bar s_{ij}}.
  \end{displaymath}

$\mathrm{2)}$	If $\sum_{i,j=1}^n\bar s_{ij} \leq n$ , then 
 \begin{displaymath}
 \mathrm{dim}_BA=2.
  \end{displaymath}
 \end{Thot}
\begin{proof}
The proof is similar to one in \cite{yun1}. Look at the Theorem 1 and Remark 2.
\end{proof}


\end{document}